\documentclass[11pt,reqno]{amsart}

\usepackage[left=3.5cm,right=3.5cm,top=3.5cm,bottom=3.5cm]{geometry}
\usepackage{tipa}
\usepackage{amssymb}
\usepackage{graphicx}
\usepackage[hidelinks]{hyperref}
\usepackage{enumerate}
\usepackage{xcolor}
\usepackage{amsmath}
\allowdisplaybreaks[4]

\numberwithin{equation}{section}
\newtheorem{theorem}{Theorem}[section]

\newtheorem{corollary}{Corollary}[section]
\newtheorem{lemma}[theorem]{Lemma}

\theoremstyle{definition}
\newtheorem{definition}{Definition}[section]

\newtheorem*{remarks*}{Remarks}

\numberwithin{equation}{section}

\title{On the Lucas Property of Linear Recurrent Sequences}

\author[H. Zhong]{Hao Zhong}
\address{(H. Zhong) School of Mathematical Sciences, Zhejiang University, Hangzhou, 310027, China}
\curraddr{}
\email{11435011@zju.edu.cn}
\thanks{}

\author[T. Cai]{Tianxin Cai}
\address{(T. Cai) School of Mathematical Sciences, Zhejiang University, Hangzhou, 310027, China}
\curraddr{}
\email{txcai@zju.edu.cn}
\thanks{}

\keywords{Lucas property, Fibonacci sequences, Lucas numbers, linear recurrent sequences}
\subjclass[2010]{11B50, 11B39}

\begin{document}

\maketitle

\thispagestyle{empty}

\begin{abstract}
Let $S$ be an arithmetic function. $S$ has Lucas property if for any prime $p$ and $n=\sum_{i=0}^{r}n_{i}p^{i}$, where $0 \leq n_{i}\leq p-1$,
\begin{equation}\label{eq:1.1}
  S(n)\equiv S(n_{0})S(n_{1})\ldots S(n_{r})\pmod p.
\end{equation}
In this note, we discuss the Lucas property of Fibonacci sequences and Lucas numbers. Meanwhile, we find some other interesting results.
\end{abstract}

\section{Introduction}

The famous Lucas' theorem states that
\begin{equation}\label{eq:Lucs}
\binom{n}{m} \equiv \binom{n_{0}}{m_{0}}\binom{n_{1}}{m_{1}}\ldots\binom{n_{r}}{m_{r}}\pmod p,
\end{equation}
where $n,m\in\mathbb{N}$, the base $p$ expansions of $n$ and $m$ are $n=\sum_{i=0}^{r}n_{i}p^{i},m=\sum_{i=0}^{r}m_{i}p^{i}$ $(0 \leq n_{i}, m_{i} \leq p-1).$

In 1992, Richard J. McIntosh \cite{McIntosh1992} gave a definition of the Lucas property and the double Lucas property, i.e.,

\begin{definition}\label{def:1}
Let $S$ be an arithmetic function. $S$ has Lucas property if for any prime $p$ and $n=\sum_{i=0}^{r}n_{i}p^{i}$, where $0 \leq n_{i}\leq p-1$,
\begin{equation}\label{eq:1.1}
  S(n)\equiv S(n_{0})S(n_{1})\ldots S(n_{r})\pmod p.
\end{equation}

And let $D$ be an bivariate arithmetic function. $D$ has double Lucas property if for any prime $p$, $n=\sum_{i=0}^{r}n_{i}p^{i}$, and $m=\sum_{i=0}^{r}m_{i}p^{i}$, where $0 \leq n_{i}, m_{i}\leq p-1$,
\begin{equation}\label{eq:1.2}
  D(n,m)\equiv D(n_{0},m_{0})D(n_{1},m_{1})\ldots D(n_{r},m_{r})\pmod p.
\end{equation}
Another way of stating this is to say that $S$ is an LP function and $D$ is a DLP function.
\end{definition}

There are numerous examples: $a^{n}$ is an LP function for any rational number $a$; the Ap\'{e}ry numbers $A(n)=\sum_{k=0}^{n}\binom{n}{k}^{2}\binom{n+k}{k}^{2}$ is an LP function (Cf. Gessel \cite{Gessel1982}); the function $\omega(n)$ defined by
\begin{equation*}
  \frac{1}{J_{0}(2z^{1/2})}=\sum_{n=0}^{\infty}\omega(n)\frac{z^{n}}{(n!)^{2}}
\end{equation*}
is an LP function (Cf. Carlitz \cite{Carlitz1955}); and according to Lucas' theorem, the binomial coefficient $D(n,m)=\binom{n}{m}$ is a DLP function.

Moreover, we add another definition.

\begin{definition}\label{def:2}
Let $S$ be an arithmetic function. $S$ has Lucas property with the prime $p$ if for any $n=\sum_{i=0}^{r}n_{i}p^{i}$, where $0 \leq n_{i}\leq p-1$,
\begin{equation}\label{eq:1.1}
  S(n)\equiv S(n_{0})S(n_{1})\ldots S(n_{r})\pmod p.
\end{equation}
It can be said that $S$ is an LP function with the prime $p$.
\end{definition}

In this paper, we discuss the Lucas property of Fibonacci and Lucas numbers.

Let $F_{n}$ be the Fibonacci sequence, i.e.,  ${F_{n}}:F_{0}=0,F_{1}=1,F_{n}=F_{n-1}+F_{n-2}\quad(n\geq2),$ and $L_{n}$ be the Lucas numbers ${L_{n}}:L_{0}=2,L_{1}=1,L_{n}=L_{n-1}+L_{n-2}\quad(n\geq2).$

We obtain the following theorems.

\begin{theorem}\label{th:1}
Let $a,b$ be two positive integers. Then $S(n)=F_{an+b}$ is an LP function with the prime $p$ if and only if

\begin{equation}\label{eq:th1}
\begin{cases}
F_{a} & \equiv 0  \pmod p, \\
F_{b} & \equiv 1  \pmod p.
\end{cases}
\end{equation}
\end{theorem}

For Lucas numbers, we have

\begin{theorem}\label{th:2}
Let $a,b$ be two positive integers. Then $S(n)=L_{an+b}$ is an LP function with the prime $p$ if and only if

\begin{equation}\label{eq:th2}
\begin{cases}
5F_{a} & \equiv 0  \pmod p, \\
F_{b} & \equiv 1  \pmod p.
\end{cases}
\end{equation}
\end{theorem}

From these two theorems, we can obtain some corollaries.

\begin{corollary}\label{cor:1}
Let $a$ and $b$ be positive integers. Then $S(n)=F_{an+b}$ is not an LP function and $L_{an+b}$ is not an LP function.
\end{corollary}

\begin{proof}[Proof]
The proof is by contradiction. Let $a$ and $b$ be positive integers such that $S(n)=F_{an+b}$ is an LP function. Then by Theorem \ref{th:1}, $p$ divides $F_{a}$ for any prime $p$, a contradiction. A similar proof follows for $S(n)=L_{an+b}$.
\end{proof}

\begin{corollary}\label{cor:2}
Let $p=5$. Then for any positive integer $a$,

(1)$S(n)=F_{5an+b}$ is an LP function with the prime $5$, where $b \equiv 1, 2, 8$ or $19 \pmod {20}$.

(2)$S(n)=L_{an+b}$ is an LP function with the prime $5$, where $b \equiv 1 \pmod 4$.
\end{corollary}

\begin{corollary}\label{cor:3}
 Let $p$ be a Fibonacci prime, namely, there exists a positive integer $a$ such that $F_{a}=p$. Then $F_{an+1}$ is an LP function with the prime $p$ and $L_{an+1}$ is an LP function with the prime $p$.
\end{corollary}

More generally, let $\alpha(p):=\min\{n\big|p \quad \text{divides} \quad F_{n}\}$ for a prime $p$. Then we have

\begin{corollary}\label{cor:4}
The condition $F_{a} \equiv 0 \pmod p$ in Theorem \ref{th:1} can be replaced by  $a=\alpha(p)k$, where $k$ is an arbitrary positive integer. And if $p \neq 5$, the condition $5F_{a} \equiv 0 \pmod p$ in Theorem \ref{th:2} can also be replaced by  $a=\alpha(p)k$, where $k$ is an arbitrary positive integer.
\end{corollary}

\begin{proof}[Proof]
For any integers $m,n$, $gcd(F_{m},F_{n})=F_{gcd(m,n)}$. Hence, $gcd(F_{a},F_{\alpha(p)})=F_{gcd(a,\alpha(p))}$. And if $F_{a} \equiv 0 \pmod p$, then $p|F_{gcd(a,\alpha(p))}$. From the definition of $\alpha(p)$, we obtain that $gcd(a,\alpha(p))=\alpha(p)$. So, $a=\alpha(p)k$ for some integer $k$.

Similarly, for any positive integer $k$, $gcd(F_{\alpha(p)k},F_{\alpha(p)})=F_{gcd(\alpha(p)k,\alpha(p))}=F_{\alpha(p)}$. Hence, $F_{\alpha(p)k} \equiv 0 \pmod p$.
\end{proof}

A natural extension of these two theorems is to look at the Lucas property of general linear recurrent sequences. We obtain an analogous result to the two theorems above.

\begin{theorem}\label{th:3}
Let $A_{n}$ be a linear recurrent sequence, i.e., $\{A(n)\}$ satisfies the linear recurrent relation:
\begin{equation*}
A_{n}=uA_{n-1}+vA_{n-2}(n\geq2),
\end{equation*}
where $A_{0}$, $A_{1}$, $u$ and $v$ are all integers. Then for any integers $a$ and $b$, $S(n)=A_{an+b}$ is an LP function with the prime $p$ if and only if
\begin{equation}\label{eq:th3}
\begin{cases}
vs(a-1,u,v)(vA_{0}^{2}+uA_{0}A_{1}-A_{1}^{2}) & \equiv 0  \pmod p, \\
A_{b} & \equiv 1  \pmod p.
\end{cases}
\end{equation}
where
\begin{equation*}
s(k,u,v)=\sum_{i=0}^{\lfloor \frac{k}{2} \rfloor}\binom{k-i}{i}u^{k-2i}v^{i}.
\end{equation*}
\end{theorem}

When it comes to the generalizations of Fibonacci numbers, we obtain two more corollaries.

\begin{corollary}\label{cor:5}
Let $\{A_{n}\}$ be an Lucas sequence or $(P,-Q)$-Fibonacci sequence, that is, $A_{0}=0$, $A_{1}=1$ $u=P$, and $v=-Q$. Then for any integers $a$ and $b$, $S(n)=A_{an+b}$ is an LP function with the prime $p$ if and only if
\begin{equation}\label{eq:cor5.1}
\begin{cases}
Qs(a-1,P,-Q)A_{1}^{2} & \equiv 0  \pmod p, \\
A_{b} & \equiv 1  \pmod p.
\end{cases}
\end{equation}
In particular, when $\{A_{n}\}$ are Pell numbers, $S(n)=A_{an+b}$ is an LP function with the prime $p$ if and only if
\begin{equation}\label{eq:cor5.2}
\begin{cases}
s(a-1,2,1) & \equiv 0  \pmod p, \\
A_{b} & \equiv 1  \pmod p.
\end{cases}
\end{equation}
\end{corollary}

Another famous generalization of Fibonacci numbers is Fibonacci word, which is in the case of $u=v=1$. Similarly, we have
\begin{corollary}\label{cor:6}
Let $\{A_{n}\}$ be Fibonacci words. Then for any integers $a$ and $b$, $S(n)=A_{an+b}$ is an LP function with the prime $p$ if and only if
\begin{equation}\label{eq:cor6}
\begin{cases}
F_{a}(A_{0}^{2}+A_{0}A_{1}-A_{1}^{2}) & \equiv 0  \pmod p, \\
A_{b} & \equiv 1  \pmod p.
\end{cases}
\end{equation}
where $F_{a}$ is the $a$th Fibonacci number.
\end{corollary}

\section{Preliminaries}

For a fixed prime $p$, the following two corollaries from McIntosh \cite{McIntosh1992} will be needed.

\begin{lemma}\label{lem:1}
Let $S(n)$ be an LP function with the prime $p$, which is not identically zero. Then $S(0) \equiv 1 \pmod p$.
\end{lemma}

\begin{lemma}\label{lem:2}
$S(n)$ is an LP function with the prime $p$, and $S(n)$ is periodic modulo $p$ if and only if $S(n) \equiv S(1)^{n} \pmod p$.
\end{lemma}

Meanwhile, we can get the following lemma by induction on $n$.

\begin{lemma}\label{lem:3}
Let $n$ be a positive integer. Then

(1)
\begin{equation}\label{eq:2.1}
 F_{n} \equiv n3^{n-1} \pmod 5.
\end{equation}

(2)
\begin{equation}\label{eq:2.2}
 L_{n} \equiv 3^{n-1} \pmod 5.
\end{equation}
\end{lemma}

\begin{remarks*}
By using Lemma \ref{lem:2} and Lemma \ref{lem:3}, we can find some LP functions with the prime $5$,

(1) $S(n)=F_{5n+b}$ is an LP function with the prime $5$, where $b \equiv 1, 2, 8$ or $19 \pmod {20}$.

(2) $S(n)=L_{n+1}$ is an LP function with the prime $5$.
\end{remarks*}

In order to get the theorems, we need one more lemma.

\begin{lemma}\label{lem:4}
Let $n,r$ be two integers. Then

(1) (Catalan's identity)
\begin{equation}\label{eq:2.3}
 F_{n}^{2}-F_{n+r}F_{n-r}=(-1)^{n-r}\cdot F_{r}^{2}.
\end{equation}

(2)
\begin{equation}\label{eq:2.4}
 L_{n+r}L_{n-r}-L_{n}^{2}=(-1)^{n-r}\cdot 5F_{r}^{2}.
\end{equation}

(3)
\begin{equation}\label{eq:2.5}
 A_{n+r}A_{n-r}-A_{n}^{2}=(-v)^{n-r}s^{2}(r-1,u,v)(vA_{0}^{2}+uA_{0}A_{1}-A_{1}^{2}).
\end{equation}

\end{lemma}

\begin{proof}[Proof of \eqref{eq:2.4}]
We prove it by using the determinant of the matrix and the fact that
\begin{displaymath}
\left\{ \begin{array}{ll}
L_{n+r} & =F_{r+1}L_{n}+F_{r}L_{n-1},\\
L_{n} & =F_{r+1}L_{n-r}+F_{r}L_{n-r-1}.
\end{array} \right.
\end{displaymath}
Hence,
\begin{eqnarray*}
  L_{n+r}L_{n-r}-L_{n}^{2} &=& {\left|\begin{array}{cc}L_{n+r} & L_{n}\\ L_{n} & L_{n-r} \end{array}\right|} \\
   &=& {\left|\begin{array}{cc}F_{r+1}L_{n}+F_{r}L_{n-1} & L_{n}\\ F_{r+1}L_{n-r}+F_{r}L_{n-r-1} & L_{n-r} \end{array}\right|} \\
   &=& F_{r}{\left|\begin{array}{cc}L_{n-1} & L_{n}\\ L_{n-r-1} & L_{n-r} \end{array}\right|}\\
   &=& F_{r}{\left|\begin{array}{cc}L_{n-1} & L_{n-2}\\ L_{n-r-1} & L_{n-r-2} \end{array}\right|}\\
   &=& \cdots \\
   &=& (-1)^{n-r}F_{r}{\left|\begin{array}{cc}L_{r+1} & L_{r}\\ L_{1} & L_{0} \end{array}\right|}\\
   &=& (-1)^{n-r}F_{r}(2L_{r+1}-L_{r})\\
   &=& (-1)^{n-r}F_{r}(L_{r+1}+L_{r-1})\\
   &=& (-1)^{n-r}\cdot 5F_{r}^{2}.
\end{eqnarray*}
So, \eqref{eq:2.4} is true.
\end{proof}

\begin{proof}[Proof of \eqref{eq:2.5}]
To prove \eqref{eq:2.5}, we first prove that
\begin{equation}\label{eq:2.6}
A_{n+r}=s(k,u,v)A_{n+r-k}+t(k,u,v)A_{n+r-k-1},
\end{equation}
where $s(k,u,v)=\sum_{i=0}^{\lfloor \frac{k}{2} \rfloor}\binom{k-i}{i}u^{k-2i}v^{i}$ and $t(k,u,v)=\sum_{j=0}^{\lfloor \frac{k-1}{2} \rfloor}\binom{k-1-j}{j}u^{k-1-2j}v^{j+1}$. For k=1, \eqref{eq:2.6} holds. By inducting on $k$, we can obtain the result. Assume for $k=1,2,\ldots,m$, \eqref{eq:2.6} holds. For $k=m+1$,
\begin{align*}
&A_{n+r}=s(m,u,v)A_{n+r-m}+t(m,u,v)A_{n+r-m-1}\\
&= \sum_{i=0}^{\lfloor \frac{m}{2} \rfloor}\binom{m-i}{i}u^{m-2i}v^{i}A_{n+r-m}+\sum_{j=0}^{\lfloor \frac{m-1}{2} \rfloor}\binom{m-1-j}{j}u^{m-1-2j}v^{j+1}A_{n+r-m-1}\\
&= \sum_{i=0}^{\lfloor \frac{m}{2} \rfloor}\binom{m-i}{i}u^{m-2i}v^{i}\Big(uA_{n+r-m-1}+vA_{n+r-m-2}\Big)+\sum_{j=0}^{\lfloor \frac{m-1}{2} \rfloor}\binom{m-1-j}{j}u^{m-1-2j}v^{j+1}A_{n+r-m-1}\\
&= (\sum_{i=0}^{\lfloor \frac{m}{2} \rfloor}\binom{m-i}{i}u^{m+1-2i}v^{i}+\sum_{j=0}^{\lfloor \frac{m-1}{2} \rfloor}\binom{m-1-j}{j}u^{m-1-2j}v^{j+1})A_{n+r-m-1}\\&+\sum_{i=0}^{\lfloor \frac{m}{2} \rfloor}\binom{m-i}{i}u^{m-2i}v^{i+1}A_{n+r-m-2}\\
&= \Big(\sum_{i=0}^{\lfloor \frac{m}{2} \rfloor}\binom{m-i}{i}u^{m+1-2i}v^{i}+\sum_{j=0}^{\lfloor \frac{m-1}{2} \rfloor}\binom{m-1-j}{j}u^{m-1-2j}v^{j+1}\Big)A_{n+r-m-1}+t(m+1,u,v)A_{n+r-m-2}
\end{align*}.
\\
If $m \equiv 0 \pmod 2$,
\begin{align*}
&\sum_{i=0}^{\lfloor \frac{m}{2} \rfloor}\binom{m-i}{i}u^{m+1-2i}v^{i}+\sum_{j=0}^{\lfloor \frac{m-1}{2} \rfloor}\binom{m-1-j}{j}u^{m-1-2j}v^{j+1}\\
&= \sum_{i=0}^{\frac{m}{2}}\binom{m-i}{i}u^{m+1-2i}v^{i}+\sum_{j=0}^{\frac{m}{2}-1}\binom{m-1-j}{j}u^{m-1-2j}v^{j+1}\\
&= u^{m+1}+\sum_{i=1}^{\frac{m}{2}}(\binom{m-i}{i}+\binom{m-i}{i-1})u^{m+1-2i}v^{i}\\
&= u^{m+1}+\sum_{i=1}^{\frac{m}{2}}{\binom{m+1-i}{i}}u^{m+1-2i}v^{i}\\
&= s(m+1,u,v).
\end{align*}

If $m \equiv 1 \pmod 2$,
\begin{align*}
&\sum_{i=0}^{\lfloor \frac{m}{2} \rfloor}\binom{m-i}{i}u^{m+1-2i}v^{i}+\sum_{j=0}^{\lfloor \frac{m-1}{2} \rfloor}\binom{m-1-j}{j}u^{m-1-2j}v^{j+1}\\
&= \sum_{i=0}^{\frac{m-1}{2}}\binom{m-i}{i}u^{m+1-2i}v^{i}+\sum_{j=0}^{\frac{m-1}{2}}\binom{m-1-j}{j}u^{m-1-2j}v^{j+1}\\
&= v^{\frac{m+1}{2}}+\sum_{i=0}^{\frac{m-1}{2}}(\binom{m-i}{i}+\binom{m-i}{i-1})u^{m+1-2i}v^{i}\\
&= v^{\frac{m+1}{2}}+\sum_{i=0}^{\frac{m-1}{2}}{\binom{m+1-i}{i}}u^{m+1-2i}v^{i}\\
&= s(m+1,u,v).
\end{align*}

Hence, $A_{n+r}=s(m+1,u,v)A_{n+r-m-1}+t(m+1,u,v)A_{n+r-m-2}$, which means \eqref{eq:2.6} holds. The rest of the proof is similar to the proof of \eqref{eq:2.4}. By using the determinant of the matrix, we can obtain
\begin{equation*}
A_{n+r}A_{n-r}-A_{n}^{2}=(-1)^{n-r}v^{n-r-1}t(r,u,v)(A_{r+1}A_{0}-A_{r}A_{1}).
\end{equation*}
By using \eqref{eq:2.6} and the fact $t(r,u,v)=vs(r-1,u,v)$, we have
\begin{align*}
&A_{n+r}A_{n-r}-A_{n}^{2}=(-1)^{n-r}v^{n-r-1}t(r,u,v)(A_{r+1}A_{0}-A_{r}A_{1})\\
&= (-1)^{n-r}v^{n-r-1}t(r,u,v)(t(r,u,v)A_{0}^{2}-s(r-1,u,v)A_{1}^{2}+s(r,u,v)A_{1}A_{0}-t(r-1,u,v)A_{0}A_{1})\\
&= (-v)^{n-r}s(r-1,u,v)(vs(r-1,u,v)A_{0}^{2}-s(r-1,u,v)A_{1}^{2}+s(r,u,v)A_{1}A_{0}-vs(r-2,u,v)A_{0}A_{1})\\
&= (-v)^{n-r}s^{2}(r-1,u,v)(vA_{0}^{2}+uA_{0}A_{1}-A_{1}^{2}).
\end{align*}
So, \eqref{eq:2.5} is true.

\end{proof}

\section{Proofs of the theorems}

\begin{proof}[Proof of Theorem \ref{th:1}]
The Fibonacci numbers are periodic modulo $p$ for any prime $p$. So is $S(n)=F_{an+b}$, where $a,b$ are positive integers.

We first prove the necessity. Assume that $S(n)=F_{an+b}$ is an LP function with the prime $p$. From Lemma \ref{lem:1}, $S(0) \equiv 1 \pmod p$, so $F_{b} \equiv 1 \pmod p$. And from Lemma \ref{lem:2}, for any positive integer $n$, $F_{an+b} \equiv F_{a+b}^{n} \pmod p$. Set $n=2$, $F_{2a+b} \equiv F_{a+b}^{2} \pmod p$. By using Catalan's identity \eqref{eq:2.3}, we have
\begin{align*}
  F_{a+b+a}F_{a+b-a} &= F_{a+b}^{2}-(-1)^{a+b-a}F_{a}^{2}\\
  F_{2a+b}F_{b} &= F_{a+b}^{2}-(-1)^{b}F_{a}^{2}\\
  F_{a+b}^{2} &\equiv F_{a+b}^{2}-(-1)^{b}F_{a}^{2} \pmod p\\
  F_{a}^{2} &\equiv 0 \pmod p\\
  F_{a} &\equiv 0 \pmod p.
\end{align*}
Hence, $a$ and $b$ satisfy
\begin{displaymath}
\left\{ \begin{array}{ll}
F_{a} & \equiv 0 \pmod p\\
F_{b} & \equiv 1 \pmod p
\end{array} \right.
\end{displaymath}

Next we prove the sufficiency. From Lemma \ref{lem:2}, we have to prove that
\begin{equation}\label{eq:3.1}
S(n) \equiv S(1)^{n} \pmod p.
\end{equation}
And we'll prove it by induction on $n$. For $n=1$, it's obviously true. Assume that for $n\leq k$, \eqref{eq:3.1} holds. For $n=k+1$, by using Catalan's identity \eqref{eq:2.3}, we have
\begin{align*}
  F_{ak+b+a}F_{ak+b-a} &= F_{ak+b}^{2}-(-1)^{ak+b-a}F_{a}^{2}\\
  F_{a(k+1)+b}F_{a(k-1)+b} &= F_{ak+b}^{2}-(-1)^{a(k-1)+b}F_{a}^{2}\\
  F_{a(k+1)+b}F_{a+b}^{k-1} &\equiv F_{a+b}^{2k}-(-1)^{a(k-1)+b}F_{a}^{2} \pmod p\\
  F_{a(k+1)+b} &\equiv F_{a+b}^{k+1} \pmod p.
\end{align*}
Hence, \eqref{eq:3.1} holds for any positive integer $n$. And $F_{an+b}$ is an LP function with the prime $p$.
\end{proof}

\begin{proof}[Proof of Theorem \ref{th:2}]

The proof is similar to Theorem \ref{th:1}. Lucas number is periodic modulo $p$ for any prime $p$. So is $S(n)=L_{an+b}$, where $a,b$ are positive integers.
We first prove the necessity. Assume that $S(n)=L_{an+b}$ is an LP function with the prime $p$. From Lemma \ref{lem:1}, $S(0) \equiv 1 \pmod p$, so $L_{b} \equiv 1 \pmod p$. And from Lemma \ref{lem:2}, for any positive integer $n$, $L_{an+b} \equiv L_{a+b}^{n} \pmod p$. Set $n=2$, $L_{2a+b} \equiv L_{a+b}^{2} \pmod p$. By using \eqref{eq:2.4}, we have
\begin{align*}
  L_{a+b+a}L_{a+b-a} &= L_{a+b}^{2}+(-1)^{a+b-a}\cdot 5F_{a}^{2}\\
  L_{2a+b}L_{b} &= L_{a+b}^{2}+(-1)^{b}\cdot 5F_{a}^{2}\\
  L_{a+b}^{2} &\equiv L_{a+b}^{2}+(-1)^{b}\cdot 5F_{a}^{2} \pmod p\\
  5F_{a}^{2} &\equiv 0 \pmod p\\
  5F_{a} &\equiv 0 \pmod p.
\end{align*}
Hence, $a$ and $b$ satisfy
\begin{displaymath}
\left\{ \begin{array}{ll}
5F_{a} & \equiv 0 \pmod p\\
L_{b} & \equiv 1 \pmod p
\end{array} \right.
\end{displaymath}

Next we prove the sufficiency. From Lemma \ref{lem:2}, we also have to prove that \eqref{eq:3.1} is true. And we'll prove it by induction on $n$. For $n=1$, it's obviously true. Assume that for $n\leq k$, \eqref{eq:3.1} holds. For $n=k+1$, by using \eqref{eq:2.4}, we have
\begin{align*}
  L_{ak+b+a}L_{ak+b-a} &= L_{ak+b}^{2}+(-1)^{ak+b-a}\cdot 5F_{a}^{2}\\
  L_{a(k+1)+b}L_{a(k-1)+b} &= L_{ak+b}^{2}+(-1)^{a(k-1)+b}\cdot 5F_{a}^{2}\\
  L_{a(k+1)+b}L_{a+b}^{k-1} &\equiv L_{a+b}^{2k}+(-1)^{a(k-1)+b}\cdot 5F_{a}^{2} \pmod p\\
  L_{a(k+1)+b} &\equiv L_{a+b}^{k+1} \pmod p.
\end{align*}
Hence, \eqref{eq:3.1} holds for any positive integer $n$. And $L_{an+b}$ is an LP function with the prime $p$.
\end{proof}

\begin{proof}[Proof of Theorem \ref{th:3}]

From \cite{Carmichael1920} and \cite{Robinson1966}, we know that for any integer $m$, a linear recurrent sequence of integers modulo $m$ is periodic. The same is true for a prime $p$. Hence $S(n)=A_{an+b}$ is periodic modulo $p$. To obtain the proof it is enough to apply the reasoning just like in the proofs of Theorem \ref{th:1} and Theorem \ref{th:2}.
\end{proof}

\subsection*{Acknowledgments}
This work is supported by the National Natural Science Foundation of China (Grant No. 11501052 and Grant No. 11571303).

\bibliographystyle{amsplain}

\end{document}